\documentclass[10pt]{amsart}%
\usepackage{amsmath}
\usepackage{amsfonts}
\usepackage{amssymb}
\usepackage{graphicx}%
\providecommand{\U}[1]{\protect\rule{.1in}{.1in}}
\newtheorem{theorem}{Theorem}[section]
\newtheorem{acknowledgement}[theorem]{Acknowledgment}

\newtheorem{lemma}[theorem]{Lemma}

\newtheorem{proposition}[theorem]{Proposition}

\begin{document}
\title[Mass transport equations]{A McLean Theorem for the moduli space of Lie solutions to mass transport
equations. }
\author{Micah Warren}
\address{Department of Mathematics, Princeton University, Princeton NJ, USA}
\email{mww@math.princeton.edu}
\thanks{The author is supported in part by NSF Grant DMS-0901644. }
\date{\today}
\maketitle

\begin{abstract}
On compact manifolds which are not simply connected, we prove the existence of
\textquotedblleft fake\textquotedblright\ solutions to the optimal
transportion problem. These maps preserve volume and arise as the exponential
of a closed 1-form, hence appear geometrically like optimal transport maps.
The set of such solutions forms a manifold with dimension given by the first
Betti number of the manifold. In the process, we prove a Hodge-Helmholtz
decomposition for vector fields. The ideas are motivated by the analogies
between special Lagrangian submanifolds and solutions to optimal transport problems.

\end{abstract}

\section{\bigskip Introduction}

In \cite{KMW}, graphs of solutions to the optimal transportation were shown to
solve a volume maximization problem, using a calibration argument. With the
appropriate metric on the product space $M\times\bar{M}$, maximality is
equivalent to the vanishing of certain differential forms along the graph of
the optimal map. \ In this note, we discuss the converse and see that, at
least in the smooth case, topology allows for maximizers of the volume problem
which do not arise as solutions to the optimal transportation problem, despite
locally having the same geometric properties. \ These maximizers are special
Lagrangian in the sense of Hitchin \cite{H} and Mealy \cite{Me}, a pseudo
Riemannian analogue of the special Lagrangian geometry of Harvey and Lawson
\cite{HL}. \ 

In the case when the cost is given by Riemannian distance squared, Delano\"{e}
\cite{D} introduced the notion of ``Lie solutions of Riemannian transport
equations" (see also [L].) \ The graphs of Delano\"{e}'s solutions are
maximizers of the volume maximization problem discussed in \cite{KMW}.

We recall the McLean Theorem \cite[Theorem 3.6]{McL} \cite[Theorem 3.21]{Ma}

\begin{theorem}
[McClean]\ Suppose $L$ is an smooth embedded special Lagrangian submanifold of
a Calabi-Yau\ manifold. \ The moduli space $\mathcal{M}$ of special Lagrangian
submanifolds near $L$ is a manifold of dimension $b_{1}(L).$ \ The tangent
space to $\mathcal{M}$ is identified with the harmonic $1$-forms on $L,$ which
has a naturally induced $L^{2}~$metric.
\end{theorem}

\bigskip The seminal paper of Harvey and Lawson shows that special Lagrangian
submanifolds are minimal submanifolds in Calabi-Yau manifolds, and that a
manifold is special Lagrangian if and only if the K\"{a}hler form and a
certain $n$-form vanishes along the submanifold. \ Hitchin \cite{H} analyzed
the metric on the moduli space in McLean's theorem and showed that it arises
via a Lagrangian embedding into a pseudo-Riemannian space. Hitchin's notion of
special Lagrangian, that the K\"{a}hler form and certain combinations of
$n$-forms vanish, describes Lie solutions of the mass transport problems. \ 

The optimal transportation problem is the following. \ Given probability
volume forms $\rho$ and $\bar{\rho}$ on manifolds $M$ and $\bar{M}$, and a
cost function $c:$ $M\times\bar{M}\rightarrow%
\mathbb{R}
,$ find a map $T:M\rightarrow\bar{M}$ which minimizes a cost integral,%
\[
\int_{M}c(x,T(x))d\rho
\]
among all maps $T$ which preserve the volume, i.e.
\begin{equation}
T_{\ast}\rho=\bar{\rho}.\ \label{meas}%
\end{equation}
The work of Brenier \cite{B} and McCann \cite{McC} shows that given standard
conditions on the cost function, the unique solution will be the map
satisfying (\ref{meas}) and arising as the cost exponential of the gradient of
a potential function $u:$%
\begin{equation}
T(x)=c\text{-}\exp_{x}du(x). \label{cexpdu}%
\end{equation}
\ A local version of (\ref{meas}), namely,
\begin{equation}
T^{\ast}\bar{\rho}=\rho\ \label{pullback}%
\end{equation}
is equivalent to the form%
\[
\bar{\rho}d\bar{x}-\rho dx
\]
vanishing on the graph $(x,T(x))\subset M\times\bar{M},$ and (\ref{cexpdu}%
)\ implies the vanishing of a K\"{a}hler form defined in \cite{KM}. \ The more
general problem we attack here is to find maps which locally solve the optimal
transport problem, that is, satisfy (\ref{pullback})\ and arise as the
exponential of a closed form
\begin{equation}
T(x)=c\text{-}\exp_{x}\left(  \eta(x)\right)  ,\text{ \ }d\eta=0. \label{Lie}%
\end{equation}
\ \ 

\ The result is the following. \ 

\begin{theorem}
\label{main thm} Suppose $M,\bar{M}$ are compact manifolds with nowhere
vanishing smooth densities $\rho$ and $\bar{\rho}$. \ Let $c$ be a continuous
cost function on $M\times\bar{M},$ which away from a cut locus $\mathcal{C}$,
is smooth and satisfies local and global twist assumptions (A2) and (A1). \ If
a diffeomorphism $T$ is a smooth Lie solution of the mass transport equation
(i.e. satisfying (\ref{pullback}) and (\ref{Lie}))\ avoiding the cut locus,
then there is a moduli space of smooth Lie solutions near $T$ which is a
smooth manifold of dimension $b_{1}(M).$ \ \ 
\end{theorem}

Remark:\ \ There are cost functions such that even for some smooth densities,
the optimal solutions will not be smooth, as was demonstrated by Loeper [Lo],
when the (A3) assumption on the cost is not satisfied. \ \ Also, in the
absence of the (A3) assumption, local $c$-convexity does not imply global, so
even on simply connected domains, Lie solutions are not necessarily solutions
to the optimal transport problem.

The core of our result is a Hodge-Helmholtz type decomposition holding at Lie
solutions, which decomposes deformations of maps into those which preserve
property (\ref{pullback}), those which preserve property (\ref{Lie}) and the
harmonic deformations, which preserve both.  The metric used is the one
induced by the linearized operator to the optimal transport equation,
multiplied by an explicit conformal factor. This metric and concomitant
Laplacian are defined in section 2, and their Hodge-Helmholtz properties are
shown in section 3. \ In section 4 we give a possible politico-economic
application of solutions. 

\begin{acknowledgement}
The author is very grateful for advice from Phillipe Delano\"{e}.
\end{acknowledgement}

\section{Preliminaries}

For $n$-dimensional manifolds $M$ and $\bar{M}$, let $c:N\subset M\times
\bar{M}\rightarrow\mathbf{R}$ be a continuous cost function which is smooth
almost everywhere, except on a set $\mathcal{C}\mathfrak{=}M\times\bar{M}-N$
which we call the \textquotedblleft cut locus." \ (The reason for this
terminology is clear if we use the distance squared function on a manifold as
the cost function.) \ \ Let $D\bar{D}c$ be the $n\times n$ matrix\ given by
(in local coordinates)
\[
\left(  D\bar{D}c\right)  _{i\bar{j}}(x,\bar{x})=\frac{\partial^{2}}{\partial
x^{i}\partial\bar{x}^{j}}c(x,\bar{x}).
\]
On $N$ we will require that
\begin{equation}
\det\left(  D\bar{D}c\right)  \neq0 \tag{A2}%
\end{equation}
which is a local version of the standard \emph{twist }condition: \ For each
$x$ and $\bar{x}$ respectively, \
\begin{equation}
\bar{x}\rightarrow Dc(x,\bar{x}) \tag{A1}%
\end{equation}
is invertible, with an inverse depending continuously on $x$. \ 

To be clear with our conventions, we recall the following Kantorovich problem:
If
\[
J(u,v)=\int_{M}(-u)d\rho+\int_{\bar{M}}vd\bar{\rho},
\]
the problem is to maximize $J$ over all $-u(x)+v(\bar{x})\leq c(x,\bar{x}).$
\ One also considers a dual problem:\ If
\[
I(\pi)=\int_{M\times\bar{M}}c(x,y)d\pi,
\]
find the minimum of $I$ over all measures $\pi$ on the product space
$M\times\bar{M}$ which have marginals $\rho$ and $\bar{\rho}.$ It is well
known (cf \cite{V}) that
\[
\sup_{-u(x)+v(\bar{x})\leq c(x,\bar{x})}J(u,v)=\inf_{\pi\epsilon\Pi(\rho
,\bar{\rho})}I(\pi).
\]
With this setup in mind we can derive the optimal map $T$ from $u$ as
follows:\ \ Suppose $\left(  x_{0},\bar{x}_{0}\right)  $ is a point where the
equality $-u(x_{0})+v(\bar{x}_{0})=c(x_{0},\bar{x}_{0})$ occurs. The function
\[
z_{\bar{x}_{0}}(x)=c(x,\bar{x}_{0})+u(x)\text{ }%
\]
must have a minimum at $x_{0}.$\ \ Then define the cost exponential
$T(x_{0},du)=\bar{x}_{0}$. \ If differentiable, from the fact that $z_{\bar
{x}_{0}}(x)$ is at a minimum we have%
\begin{equation}
u_{i}+c_{i}(x,T(x,du))=0\label{eq:contact}%
\end{equation}
where $c_{i}$ refers to differentiation in the first variable. \ (One can
check that $T(x,u)=-\exp_{x}\nabla u$ when $c(x,\bar{x})=d^{2}(x,\bar{x})/2$).
\ This only depends locally on the function $u,$ and clearly requires that
$Du$ stay inside the range of $Dc(x,\cdot).$ \ \ The elliptic optimal
transportation equation can be derived by taking another derivative and then a
determinant:%
\begin{align}
u_{ij}+c_{ij}+c_{i\bar{s}}T_{j}^{\bar{s}} &  =0\label{id0}\\
\det(u_{ij}+c_{ij}) &  =\det(-c_{i\bar{s}}T_{j}^{\bar{s}})=\det(-c_{i\bar{s}%
})\frac{\rho(x)}{\bar{\rho}(T(x))}\label{eq:elliptic}%
\end{align}
using the fact that $T$ locally pushes $\rho$ forward to $\bar{\rho},$ so
satisfies
\begin{equation}
\bar{\rho}(T(x))\det DT=\rho(x).\label{mp}%
\end{equation}
Here and in the sequel we use the following conventions:\ Indices $i,j,k$ etc
will be coordinates on $M$ while $\bar{p},\bar{s},\bar{v},$ etc will be
coordinates on $\bar{M}.$The variable $t$ will be reserved for variations.$\ $%
\ \ We assume on any local product chart that $c_{i\bar{s}}$ is negative
definite, and let $b_{i\bar{s}}=-c_{i\bar{s},}$ with $b^{\bar{s}i\text{ }}%
$\ its inverse. \ The tangent space is denoted $\mathfrak{T}M.$ \ We use
$\ w_{ij}(x)=u_{ij}(x)+c_{ij}(x,T(x)).$ \ \ To be clear, we will use
$\frac{\partial}{\partial x^{k}}w_{ij}$ to denote derivatives of $w$,
otherwise a subscript will denote differentiation. \ We note that most of the
quantities we deal with are coordinate invariant. \ In particular $w_{ij}$
actually is an honest tensor. \ For future use, we make note of the following
restatement of the above identity (\ref{id0})%
\begin{align}
w_{ij} &  =b_{i\bar{s}}T_{j}^{\bar{s}}\label{id1}\\
w^{kj}T_{j}^{\bar{s}} &  =b^{\bar{s}k}.\nonumber
\end{align}

\subsection{ Linearizing the elliptic equation.}

In order to linearize (\ref{eq:elliptic}) take a variation,
\begin{equation}
u(x)+tv(x).\label{variation}%
\end{equation}
First, we insert (\ref{variation}) into (\ref{eq:contact}) and differentiate
to obtain%
\[
v_{i}-b_{i\bar{s}}(x,T(x,du))D_{t}T^{\bar{s}}=0
\]
hence
\[
D_{t}T^{\bar{s}}=T_{t}^{\bar{s}}=v_{i}b^{\bar{s}i}.
\]
Now, taking a logarithm of (\ref{eq:elliptic})
\begin{align*}
F(x,Du,D^{2}u) &  =\\
&  \ln\det(u_{ij}+c_{ij}(x,T(x))-\ln\det(b_{i\bar{s}}(x,T(x))-\ln\rho
(x)+\ln\bar{\rho}(T(x))
\end{align*}
which is linearized
\begin{align}
Lv &  =\frac{d}{dt}F(u+tv)=w^{ij}(v\,_{ij}+c_{ij\bar{s}}T_{t}^{\bar{s}%
})-b^{\bar{s}i}b_{\bar{s}i\bar{p}}T_{t}^{\bar{p}}+\left(  \ln\bar{\rho
}\right)  _{\bar{s}}T_{t}^{\bar{s}}\label{eq:linearized}\\
&  =w^{ij}v\,_{ij}-w^{ij}b_{ij\bar{s}}b^{\bar{s}k}v_{k}-b^{\bar{s}i}b_{\bar
{s}i\bar{p}}b^{pk}v_{k}+\left(  \ln\bar{\rho}\right)  _{\bar{s}}b^{\bar{s}%
k}v_{k}\nonumber
\end{align}
A version of this linearized operator was introduced by Trudinger and Wang
\cite[2.18]{TW}.

\subsection{The KM and modified KM\ metrics and a related Laplace-Beltrami}

In \cite{KM} Kim and McCann considered the following pseudo metric on the
product space $M\times\bar{M}:$%
\[
h=\left(
\begin{array}
[c]{cc}%
0 & b_{i\bar{s}}\\
b_{i\bar{s}}^{T} & 0
\end{array}
\right)
\]
and symplectic form
\[
\omega=b_{i\bar{s}}dx^{i}\wedge d\bar{x}^{\bar{s}}.
\]
In this metric, the graph $(x,T(x))$ of the cost exponential of any locally
$c$-convex potential $u$ is space-like and Lagrangian. \ The induced metric,
in terms of coordinates on $M$, is given by
\[
(x,T(x))^{\ast}h=u_{ij}+c_{ij}=w_{ij}.
\]
For given mass densities $\rho$ and $\bar{\rho}$, consider the following
metric in \cite{KMW}%

\[
h=\frac{1}{2}\left(  \frac{\rho\bar{\rho}}{\det b_{is}}\right)  ^{1/n}\left(
\begin{array}
[c]{cc}%
0 & b_{i\bar{s}}\\
b_{i\bar{s}}^{T} & 0
\end{array}
\right)
\]
and the calibrating form
\[
\Omega=\frac{1}{2}(\rho+\bar{\rho}).
\]
The graph of any solution to the optimal transportation problem will be a
calibrated maximal Lagrangian surface with respect to this metric, but the
converse is not true in general. \ Calibrated maximal Lagrangian surfaces are
what we are studying in this paper. \ 

To begin, we define another metric. \ \ We will show that with respect to this
metric, the tangent space of deformations of calibrated submanifolds coincide
with harmonic $1$-forms, which is the same situation that occurs in McLean's
theorem. \ \ The metric we use is yet another conformal factor of the metric
used by Kim and McCann, differing by a power of the conformal factor from the
metric associated to the calibration in \cite{KMW}. \ \ For a given frame on
$M$, and a solution of the problem $T$ define
\begin{equation}
g_{ij}(x)=w_{ij}(x)\left(  \frac{\rho(x)\bar{\rho}(T(x))}{\det b_{i\bar{s}%
}(x,T(x))}\right)  ^{1/(n-2)}.\label{themetric}%
\end{equation}
This metric appears naturalyl suited for our problem, as we will use the Hodge
decomposition with respect to this metric in the sequel.

Before proving the next claim, note that at any point $x\in M,$ one has by the
twist condition that for any covector $\eta\in\mathfrak{T}_{x}^{\ast}M$ there
is at most one point $\bar{x}=T(x,\eta)$ in $\bar{M}$ so that $\eta
=-D(x,T(x,\eta))$ \ This defines $T$ whenever $\eta$ is in an appropriate
domain. \ 

Our first claim is the following.\ 

\begin{proposition}
{\label{LB}} Let $n\geq3.$ \ Let $T=T(x,\eta)$ be a cost exponential of a
$1$-form locally given by $du.$ \ \ Define%
\[
\theta=\ln\det w_{ij}-\ln\rho(x)+\ln\bar{\rho}(x,T(x))-\ln\det b_{i\bar{s}},
\]
and
\[
\lambda=\left(  \frac{\rho(x)\bar{\rho}(T(x))}{\det b_{i\bar{s}}%
(x,T(x))}\right)  ^{1/(n-2)},
\]
then with the above metric (\ref{themetric}) on $M,$
\[
Lz=\lambda\left(  \triangle_{g}z+\frac{1}{2}\langle\nabla\theta,\nabla
z\rangle_{g}\right)
\]
where $\triangle_{g}z$ is Laplace-Beltrami operator with respect to $g$ and
$L$ is a the linearized operator defined by (\ref{eq:linearized}) . \ 
\end{proposition}

\begin{proof}
\ Defining
\[
g_{ij}=\lambda w_{ij}%
\]
with
\[
\lambda=\left(  \frac{\rho(x)\bar{\rho}(T(x))}{\det b_{is}(x,T(x))}\right)
^{1/(n-2)}%
\]
we compute in local coordinates:%
\begin{align*}
\triangle_{g}z &  =\frac{1}{\lambda^{n/2}\sqrt{\det w_{ij}}}(\lambda
^{n/2-1}\sqrt{\det w_{ij}}w^{ij}z_{i})_{j}\\
&  =\frac{1}{\lambda}w^{ij}z_{ij}+\frac{1}{\lambda}(-w^{ik}w^{lj}%
\frac{\partial}{\partial x^{j}}w_{kl}+\frac{n-2}{2}w^{ij}\left(  \ln
\lambda\right)  _{j}+\frac{1}{2}w^{kl}\frac{\partial}{\partial x^{j}}%
w_{kl}w^{ij})z_{i}.
\end{align*}
Locally, because $\eta=du,$ we can write $w_{ab}=u_{ab}+c_{ab}$, thus
\[
\lambda\triangle_{g}z=w^{ij}z_{ij}+(-w^{ik}w^{lj}(\frac{\partial}{\partial
x^{j}}w_{kl}-\frac{\partial}{\partial x^{k}}w_{lj})-w^{ik}w^{lj}\frac
{\partial}{\partial x^{k}}w_{lj}+\frac{n-2}{2}w^{ij}\left(  \ln\lambda\right)
_{j}+\frac{1}{2}w^{kl}\frac{\partial}{\partial x^{j}}w_{kl}w^{ij})z_{i}%
\]%
\begin{align}
&  =w^{ij}z_{ij}+\left\{  w^{ik}b_{kl\bar{s}}b^{\bar{s}l}-b_{lj\bar{s}}%
b^{\bar{s}i}w^{lj}-\frac{1}{2}w^{ij}\left(  \ln\det w\right)  _{j}+\frac
{n-2}{2}w^{ij}\left(  \ln\lambda\right)  _{j}\right\}  z_{i}\nonumber\\
&  =w^{ij}z_{ij}+\left[
\begin{array}
[c]{c}%
w^{ik}b_{kl\bar{s}}b^{\bar{s}l}-b_{lj\bar{s}}b^{\bar{s}i}w^{lj}\\
-\frac{1}{2}w^{ij}\left(  \theta_{j}-2\frac{\bar{\rho}_{\bar{s}}T_{j}^{\bar
{s}}}{\bar{\rho}}+2b^{\bar{s}k}(b_{k\bar{s}j}+b_{k\bar{s}\bar{p}}T_{j}%
^{\bar{p}})\right)
\end{array}
\right]  z_{i}\label{line3}\\
&  =w^{ij}z_{ij}+\left\{  -b_{lj\bar{s}}b^{\bar{s}i}w^{lj}+\frac{\bar{\rho
}_{\bar{s}}}{\bar{\rho}}b^{\bar{s}i}-\frac{1}{2}w^{ij}\theta_{j}-b^{\bar{s}%
k}b_{k\bar{s}\bar{p}}b^{\bar{p}i})\right\}  z_{i}\nonumber\\
&  =Lz-\frac{1}{2}\lambda\langle\nabla\theta,\nabla z\rangle_{g}\nonumber
\end{align}
where we have used (\ref{id1}) repeatedly. \ In (\ref{line3}) we combined the
following two identities \
\[
\ln\lambda=\frac{1}{(n-2)}\left\{  \ln\rho(x)+\ln\bar{\rho}(x,T(x))-\ln\det
b_{is}\right\}
\]%
\[
\ln\det w=\theta+\ln\rho(x)-\ln\bar{\rho}(x,T(x))+\ln\det b_{is}%
\]
before taking the derivative. \ 
\end{proof}

\section{\bigskip Deformations}

Let $T:M\rightarrow\bar{M}$ be a Lie solution of the transport equation. \ As
was used in the previous calculation, this means that $T$ is locally a
solution of the optimal transport equation, or equivalently, $T$ is a cost
exponential of a closed $1$ form, satisfying%

\[
\det DT=\frac{\rho(x)}{\bar{\rho}(T(x))}.
\]
Throughout this section we will use the metric $g$ on $M$ defined by
(\ref{themetric}).

Given a $1$-form $\eta$ on $M$ we can define a (vertical)\ deformation vector
field
\[
V=-\eta_{i}c^{\bar{s}i}\partial_{\bar{s}}%
\]
where $\partial_{\bar{s}}$ is the coordinate tangent frame for\ $\mathfrak{T}%
\bar{M}$ \ \ (to be precise, $V$ is a section of the pullback bundle $T^{\ast
}\mathfrak{T}\bar{M}).$ \ For a vector field in this bundle, define a
deformation of the map $T$ via
\[
T_{V}(x)=\exp_{T(x)}^{\bar{M}}(V(x)).
\]
\ (We have fixed arbitrarily a metric on $\bar{M}$ in order to define
$\exp^{\bar{M}}$, which is no more than a convenience.) \ 

Along $N=$ $M\times\bar{M}-\mathcal{C}$ we define the $1$-form $\sigma
=c_{i}(x,\bar{x})dx^{i}$ (which differs from the total differential of $c$ by
$c_{\bar{s}}(x,\bar{x})d\bar{x}^{\bar{s}}$ ). \ Then we get the following
exact K\"{a}hler form
\[
\omega=d\sigma=-c_{i\bar{s}}(x,\bar{x})dx^{i}\wedge d\bar{x}^{\bar{s}}%
\]
on $N.$

At a solution $T$ we define the map
\[
\Phi:\Lambda^{1}(M)\rightarrow\Lambda^{2}(M)\oplus\Lambda^{0}(M)
\]
via%
\[
\Phi(\eta)=((Id\times T_{V})^{\ast}\omega,\ast\left(  T_{V}(x)^{\ast}\bar
{\rho}d\bar{x}-\rho dx\right)  )
\]

Note that the level set $\Phi^{-1}(0,0)$ consists of forms whose corresponding
maps $T_{V}$ are Lie solutions. \ \ Equivalently, the map $(Id\times T_{V})$
is a calibrated Lagrangian submanifold of $M\times\bar{M}$ with respect to the
metric in \cite{KMW}.

\begin{lemma}
\bigskip The image of $\Phi$ lies in exact $2$-forms and coexact $0$-forms. \ 
\end{lemma}

\begin{proof}
The first factor is a pullback of an exact form. For the second factor,
\[
\int T_{V}(x)^{\ast}\bar{\rho}d\bar{x}-\rho dx=0,
\]
which follows easily from the fact that the diffeomorphism $T$ is simply a
change of integration variables. We are assuming the total mass of both
densities is $1.$
\end{proof}

\begin{lemma}
\label{main lemma} At $0\in C^{k+1,\alpha}(\Lambda^{1}(M))$, the differential
$D\Phi$ satisfies
\[
D\Phi(\eta)=(d\eta,d^{\ast}\eta).
\]
In particular, $D\Phi$ is a topological linear isomorphism
\[
d^{\ast}C^{k+2,\alpha}(\Lambda^{2}(M))\oplus dC^{k+2,\alpha}(\Lambda
^{0}(M))\longrightarrow dC^{k+1,\alpha}(\Lambda^{1}(M))\oplus d^{\ast
}C^{k+1,\alpha}(\Lambda^{1}(M)).
\]

\end{lemma}

\begin{proof}
To compute the derivative, for the first factor of the map $\Phi$ we use the
Lie derivative on the manifold $M\times\bar{M}$ \ (cf \cite[section
2.2.2.]{Ma})%
\begin{align*}
\frac{d}{dt}(Id\times T_{V})^{\ast}\omega|_{t=0} &  =%
\mathcal{L}%
_{\left(  -c^{\bar{s}m}\eta_{m}\partial_{\bar{s}}\right)  }(Id\times
T_{V})^{\ast}\omega\\
&  =(Id\times T_{V})^{\ast}\left[  -c^{\bar{s}m}\eta_{m}\partial_{\bar{s}%
}\lrcorner d\omega-d(c^{\bar{s}m}\eta_{m}\partial_{\bar{s}}\lrcorner
\omega)\right]  \\
&  =(Id\times T_{V})^{\ast}\left[  d(-c^{\bar{s}m}\eta_{m}\partial_{\bar{s}%
}\lrcorner\omega)\right]  \\
&  =(Id\times T_{V})^{\ast}\left[  d\left(  \eta_{m}dx^{m}\right)  \right]  \\
&  =d\eta.
\end{align*}
For the second factor, to begin we note that because the volume element is
given by
\[
Vol_{g}=\sqrt{\det w_{ij}\left(  \frac{\rho\bar{\rho}}{\det b_{i\bar{s}}%
}\right)  ^{n/n-2}}=\sqrt{\frac{\rho}{\bar{\rho}}\det b_{i\bar{s}}\left(
\frac{\rho\bar{\rho}}{\det b_{i\bar{s}}}\right)  ^{n/n-2}}=\rho\lambda,
\]
we have
\[
\ast(T_{v}^{\ast}\bar{\rho}d\bar{x}-\rho dx)=\frac{1}{\rho\lambda}\left(  \det
DT_{V}\bar{\rho}(T)-\rho\right)  .
\]
Differentiating
\begin{align}
&  \left(  \frac{1}{\rho\lambda}\det DT_{V}\bar{\rho}(T)-\frac{1}{\lambda
}\right)  ^{\prime}\nonumber\\
&  =\det DT_{V}\bar{\rho}(T)\frac{1}{\rho\lambda}\left\{  (\ln\det
DT)^{\prime}+\left(  \ln\bar{\rho}\right)  ^{\prime}-\left(  \ln\rho\right)
^{\prime}-\ln\lambda^{\prime}\right\}  +\frac{1}{\lambda}\left(  \ln
\lambda\right)  ^{\prime}.\label{prime}%
\end{align}
Noting that at $0$
\[
\det DT=\frac{\rho}{\bar{\rho}(T)}%
\]
the expression (\ref{prime}) becomes%
\[
\frac{1}{\lambda}\left\{  (\ln\det DT)^{\prime}+\left(  \ln\bar{\rho}\right)
^{\prime}-\left(  \ln\rho\right)  ^{\prime}\right\}  .
\]
Now in particular, differentiating with respect to $t$ we have
\begin{align}
&  \frac{d}{dt}\ast(T_{v}^{\ast}\bar{\rho}d\bar{x}-\rho dx)|_{t=0}=\frac
{1}{\lambda}\left[  \left(  DT^{-1}\right)  _{\bar{s}}^{j}T_{j,t}^{\bar{s}%
}+\left(  \ln\bar{\rho}\right)  _{\bar{s}}T_{t}^{\bar{s}}\right]  \nonumber\\
&  =\frac{1}{\lambda}\left[  \left(  DT^{-1}\right)  _{\bar{s}}^{j}\left(
b^{\bar{s}k}\eta_{k}\right)  _{j}+\left(  \ln\bar{\rho}\right)  _{\bar{s}%
}b^{\bar{s}k}\eta_{k}\right]  \nonumber\\
&  =\frac{1}{\lambda}\left[  b_{i\bar{s}}w^{ij}b^{\bar{s}k}\eta_{k,j}%
-b_{i\bar{s}}w^{ij}b^{\bar{p}k}b^{\bar{s}m}(b_{m\bar{p}j}+b_{m\bar{p}\bar{r}%
}b^{\bar{r}l}w_{lj})\eta_{k}+\left(  \ln\bar{\rho}\right)  _{\bar{s}}%
b^{\bar{s}k}\eta_{k}\right]  \nonumber\\
&  =\frac{1}{\lambda}\left[  w^{ij}\eta_{i,j}-w^{ij}b^{\bar{p}k}b_{i\bar{p}%
j}\eta_{k}-b^{\bar{p}k}b_{i\bar{p}\bar{r}}b^{\bar{r}i}\eta_{k}+\left(  \ln
\bar{\rho}\right)  _{\bar{s}}b^{\bar{s}k}\eta_{k}\right]  \label{delta}\\
&  =d^{\ast}\eta
\end{align}
using (\ref{id1}) repeatedly. \ The\ last equality follows from the proof of
Proposition \ref{LB}
\[
\lambda d^{\ast}\eta=w^{ij}\eta_{i,j}+\left\{  -b_{lj\bar{s}}b^{\bar{s}%
i}w^{lj}+\left(  \ln\bar{\rho}\right)  _{\bar{s}}b^{is}-b^{\bar{s}k}%
b_{k\bar{s}\bar{p}}b^{\bar{p}i}\right\}  \eta_{i}.
\]
The latter conclusion of the Lemma follows from standard algebra using the
Hodge decomposition%
\[
C^{k+1,\alpha}(\Lambda^{1}(M))=\mathcal{H}^{1}\oplus dC^{k+2,\alpha}%
(\Lambda^{0}(M))\oplus d^{\ast}C^{k+2,\alpha}(\Lambda^{2}(M)).
\]
\ 
\end{proof}

\textbf{Proof of Main Theorem when $n \geq3$} . \ \ Lemma\ \ref{main lemma}
shows that at $0$, the smooth map $\Phi$ has surjective differential
\[
D\Phi:C^{k+1,\alpha}(\Lambda^{1}(M))\longrightarrow d^{\ast}C^{k+1,\alpha
}(\Lambda^{1}(M))\oplus dC^{k+1,\alpha}(\Lambda^{1}(M))
\]
onto the product of exact and coexact forms. \ The kernel at $0$ is the
harmonic forms, which splits by the Hodge decomposition. \ It follows by the
Implicit Function Theorem (see [Ma, Thm 2.11]) that for each harmonic form
$\eta$ close enough $0$ there is a unique form $\chi(\eta)$ lying in the
orthogonal complement of the harmonic forms so that $\Phi(\eta+\chi(\eta))=0.$
\ Thus a neighborhood of $0$ in the harmonic forms on $M$ parametrizes the
moduli space near $T.$

\bigskip

\subsection{$n=2.$}

When $n=2$ we are missing the power of the Hodge theorem, which asserts that
the kernel of the operator has the same dimension as the first cohomology
class. \ We must somehow show that the set of solutions of \{(\ref{delta}%
)$=0\}$ has this same dimension. \ Forming the elliptic operator
\[
L=d^{\ast}d+\delta^{\ast}\delta
\]
where $\delta:\Lambda^{1}\rightarrow\Lambda^{0}$ is the operator defined by
(\ref{delta}), the Implicit Function Theorem arguments as used above provide
that the dimension of the moduli space near a solution is the dimension of the
kernel of $L.$

\bigskip For a given $M^{2},\bar{M}^{2} $ consider the transportation problem
of finding
\begin{equation}
F\left(  x,\theta\right)  =(T\left(  x,\theta\right)  ,\Theta\left(
x,\theta\right)  ):M^{2}\times S^{1}\rightarrow\bar{M}^{2}\times S^{1}
\label{lieprod}%
\end{equation}
minimizing the cost
\[
\tilde{c}[\left(  x,\theta\right)  ,(\bar{x},\bar{\theta})]=c(x,\bar
{x})+dist_{S^{1}}^{2}(\theta,\bar{\theta})
\]
among all maps $(F,\Theta)$ pushing $\rho\wedge d\theta$ forward to $\bar
{\rho}\wedge d\theta$. \ 

For any Lie solution $T:M^{2}$ $\rightarrow\bar{M}^{2}$ \ it is clear that
\begin{equation}
F\left(  x,\theta\right)  =(T(x),\theta+\theta_{0}) \label{lieF}%
\end{equation}
is a Lie solution to the problem (\ref{lieprod}). \ \ By Theorem
\ref{main thm} for $n=3$ there is a space of deformations of solutions, and it
has dimension equal to $1+b_{1}(M).$ \ Suffice then to show that these
deformations decompose into deformations in each factor. \ 

Let $T$ be a Lie solution on $M$ and let (\ref{lieF}) be a Lie solution on
$M\times S^{1}.$ Let $\eta$ be a $1$-form which defines a tangent vector to
the space of deformations on $M^{2}\times S^{1},$ at the solution $F,$ \ and
write
\[
\eta=\eta_{1}\omega^{1}+\eta_{2}\omega^{2}+\eta_{\theta}d\theta
\]
for some cotangent frame $\omega^{1},\omega^{2}$ for $M^{2}.$
\ Differentiating the equation
\[
d^{\ast}\eta=0
\]
in the $\theta$ direction, \ we have that%

\[
w^{ij}\eta_{i,j\theta}+\left\{  -b_{bj\bar{s}}b^{\bar{s}i}w^{bj}+\left(
\ln\bar{\rho}\right)  _{\bar{s}}b^{\bar{s}i}-b^{\bar{s}k}b_{\bar{s}k\bar{p}%
}b^{\bar{p}i}\right\}  \eta_{i},_{\theta}=0,
\]
using the fact that the warped product metric does not depend on $\theta.$
Also, because $\eta$ is a closed form, locally we have $\eta_{\theta,ij}%
=\eta_{i,j\theta}$ and $\eta_{i},_{\theta}=\eta_{\theta,i}$. \ Thus (the
honest function) $z=\eta_{\theta}$ locally satisfies an elliptic equation of
the form%

\[
w^{ij}z_{ij}+A^{i}z_{i}=0
\]
so enjoys a maximum principle. \ We conclude that on the compact manifold
$M^{2}\times S^{1},$ $\ $the $S^{1}$ $\ $component $\eta_{\theta}$ of the
deformation must be constant. \ In particular, the other two components
satisfy
\[
\delta(\eta_{1}\omega^{1}+\eta_{2}\omega^{2})=0
\]
so we conclude that the kernel of $L$ must have dimension $b_{1}(M).$

\section{A politician's optimal transportation problem}

We give the following application of solutions to the above problem. Suppose
you are the political leader of a compact manifold with non simply connected
topology, and you are in charge of solving an optimal transportation problem.
Due to forces in effect before you came into power, you discover that the
distributions which need to be paired are in fact the same (this fact is known
only to your office.) You are aware of the unspeakably high political
\textquotedblleft cost of doing nothing,\textquotedblright\ as well as the
clout of the Transportation Lobby, so proposing the trivial
identity\ transference plan is not an option. Instead, you must come up with a
plan which is nontrivial, but does not appear arbitrary. Deforming according
to a harmonic $1$-form as described in section 3, you can find a
transportation plan which has positive cost, and locally, to each supplier in
simply connected districts, is given by the cost exponential of a potential
function, so appears optimal. \ As long as no one exhibits a set of points for
which the plan is not cyclically monotone (cf \cite[Chapter 5]{V}), it will be
believed that the plan is a good one . \ If the deformation is small, such a
cycle will involve a large set of carefully selected points and will be quite
difficult to exhibit.

\end{document}